\documentclass[a4paper]{amsart}

\usepackage{hyperref}
\usepackage{tikz}
\hypersetup{colorlinks, urlcolor = blue, linkcolor = blue, citecolor = red}
\usepackage[colorinlistoftodos,bordercolor=orange,backgroundcolor=orange!20,linecolor=orange,textsize=scriptsize]{todonotes}%

\newcommand{\arxiv}[1]{\href{http://www.arXiv.org/abs/#1}{arXiv:#1}}

\usepackage{amsfonts}
\usepackage{amsmath}
\usepackage{amsthm}
\usepackage{amssymb}

\usepackage{graphicx}
\usepackage{subfig}
\usepackage{color}

\DeclareSymbolFont{bbold}{U}{bbold}{m}{n}
\DeclareSymbolFontAlphabet{\mathbbold}{bbold}
\newcommand{\zero}{\mathbbold{0}}
\newcommand{\unit}{\mathbbold{1}}

\newcommand{\maxplus}{{\mathbb{R}_{\max}}}
\newcommand{\maxtimes}{{\mathbb{T}}}

\newcommand{\ext}{\mbox{$\bigwedge$}}

\newcommand{\N}{\mathbb{N}}
\newcommand{\R}{\mathbb{R}}
\newcommand{\C}{\mathbb{C}}
\renewcommand{\S}{\mathcal{S}}

\theoremstyle{definition}
\newtheorem{defn}{Definition}
\newtheorem*{defn*}{Definition}
\theoremstyle{plain}
\newtheorem{thm}{Theorem}
\newtheorem*{thm*}{Theorem}
\newtheorem{prop}[thm]{Proposition}
\newtheorem*{prop*}{Proposition}
\newtheorem{lem}[thm]{Lemma}
\newtheorem*{lem*}{Lemma}
\newtheorem{cor}[thm]{Corollary}
\newtheorem*{cor*}{Corollary}
\theoremstyle{remark}

\newtheorem*{rmk*}{Remark}

\newtheorem*{exa*}{Example}

\DeclareMathOperator{\per}{per}
\DeclareMathOperator{\tr}{tr}

\DeclareMathOperator{\pat}{pat}
\DeclareMathOperator{\Diag}{Diag}
\DeclareMathOperator{\cvx}{cvx}
\DeclareMathOperator{\cav}{cav}
\DeclareMathOperator{\lcav}{lcav}
\DeclareMathOperator{\coef}{coef}

\DeclareMathOperator{\Log}{Log}
\newcommand{\card}[1]{\# #1}
\renewcommand{\geq}{\geqslant}
\renewcommand{\leq}{\leqslant}

\title{Tropical bounds for eigenvalues of matrices}
\author{Marianne Akian \and St\'ephane Gaubert \and Andrea Marchesini}
\address{INRIA and CMAP, \'Ecole Polytechnique, 91128 Palaiseau Cedex France}
\email[M.~Akian]{Marianne.Akian@inria.fr}
\email[S.~Gaubert]{Stephane.Gaubert@inria.fr}
\email[A.~Marchesini]{Andrea.Marchesini@polytechnique.edu}
\date{\today}
\thanks{A.~Marchesini is supported by a PhD fellowship of the doctoral school of \'Ecole Polytechnique. M.~Akian and S.~Gaubert were partially supported by the PGMO program of EDF and Fondation Math\'ematique Jacques Hadamard.}
\begin{document}

\begin{abstract}
Let $\lambda_1,\dots,\lambda_n$ denote the eigenvalues of a $n\times n$ matrix,
ordered by nonincreasing absolute value, and let $\gamma_1 \geq \dots \geq \gamma_n$ denote the tropical eigenvalues of an associated $n\times n$ matrix, obtained
by replacing every entry of the original matrix by its absolute value.
We show that for all $1\leq k\leq n$, $|\lambda_1\dots\lambda_k|
\leq C_{n,k} \gamma_1\dots\gamma_k$, where $C_{n,k}$ is a
combinatorial constant depending only on $k$ and on the pattern
of the matrix. 
This generalizes an inequality by Friedland (1986), corresponding to the special case $k=1$.
\end{abstract}

\maketitle

\section{Introduction}
\subsection{Motivation: bounds of Hadamard, Ostrowski, and P\'olya for the roots of polynomials}
In his memoir on the Graeffe method~\cite{Ostrowski1}, Ostrowski
proved the following result concerning the location of roots
of a complex polynomial of degree $n$, $p(z) = a_n z^n + \dots + a_1 z + a_0$.
Let  $\zeta_1, \cdots, \zeta_n$ denote the complex roots of $p$,
ordered by 
nonincreasing
absolute value, 
and let $\alpha_1 \geq \dots \geq \alpha_n$ denote 
what Ostrowski
called the \emph{inclinaisons num\'eriques} (``numerical inclinations'') of $p$.
These are defined as the exponentials of the opposites of the slopes of the Newton polygon obtained as the upper boundary of the convex hull of the set of points $\{(k,\log |a_k|)\mid 0\leq k\leq n\}$. 
Then, the inequalities
\begin{align}
  \frac{1}{\binom{n}{k}} \alpha_1 \dotsm \alpha_k
  \leq |\zeta_1 \dotsm \zeta_k| \leq \sqrt{\frac{{(k+1)}^{k+1}}{k^k}} \alpha_1 \dotsm \alpha_k \leq 
 \sqrt{e (k+1)} \alpha_1 \dotsm \alpha_k
\label{e-ost}
\end{align}
hold, for all $k \in [n]:=\{1,\dots,n\}$.
The upper bound in~\eqref{e-ost} originated from a work of Hadamard~\cite{hadamard1893}, who proved an inequality of the same form but with the multiplicative
constant $k+1$. Ostrowski attributed to P\'olya
the improved multiplicative constant in the upper bound of~\eqref{e-ost}, and obtained among other results the lower bound in~\eqref{e-ost}.

The inequalities of Hadamard, Ostrowski and P\'olya can be interpreted in the language of tropical geometry~\cite{viro,itenberg,sturmfels2005}). Let us associate to $p$ the 
{\em tropical polynomial function}, obtained by replacing
the sum by a max, and by taking the absolute value of every coefficient,
\[ p^{\times}(x) = \max_{0 \leq k \leq n} |a_k| x^k \enspace .
\]
The {\em tropical roots} of $p^\times$, or for short, of $p$,
are defined to be
the nondifferentiability points of the latter function,
counted with multiplicities
(precise definitions will be given in the next section).
These roots coincide with the $\alpha_i$. Then,
the log-majorization type inequalities
~\eqref{e-ost} control the distance between the multiset of roots of $p$
and the multiset of its tropical roots. 

The interest of tropical roots 
is that they are purely combinatorial objects.
They can be computed in linear time (assuming every arithmetical
operation takes a unit time) in a way which
is robust with respect to rounding errors. 

\subsection{Main result}
It is natural to ask whether tropical methods can still
be used to address other kinds of root location results,
like bounding the absolute values of the eigenvalues
of a matrix as a function of the absolute values of its
entries. We show here that this is indeed the case,
by giving an extension of the Hadamard-Ostrowski-P\'olya inequalities~\eqref{e-ost} to matrix eigenvalues. 

Thus, we consider a matrix $A\in \C^{n\times n}$,
and bound the absolute values of the eigenvalues of $A$ in terms of certain
combinatorial objects, which are {\em tropical eigenvalues}. The
latter are defined as the tropical
roots of a characteristic polynomial equation~\cite{abg04,abg04b}. They are given here
by the nondifferentiability points of the value function of
a parametric optimal assignment problem, depending only on the absolute
values of the entries of the matrix $A$. They can be computed
in $O(n^3)$ time~\cite{gassner}, in a way which is again insensitive
to rounding errors. 

Our main result (Theorem~\ref{upper_bound}) is the following:
given a complex $n \times n$ matrix $A = (a_{i,j})$ with eigenvalues $\lambda_1, \dots, \lambda_n$, ordered by nonincreasing absolute value, the inequality 
\begin{align}
	|\lambda_1 \dotsm \lambda_k| \leq \rho\big(\ext^k_{\per} (\pat A)\big) \,
\gamma_1 \dotsm \gamma_k 
\label{e-fried-gen}
\end{align}
holds for all $k \in [n]$.
Here $\gamma_1\geq \cdots \geq \gamma_n $ are the tropical eigenvalues of $A$,
$\rho$ denotes the Perron root (spectral radius),
the {\em pattern} of $A$, $\pat A$ is a $0/1$ valued matrix,
depending only on the position of non-zero entries of $A$,
and $\ext^k_{\per}(\cdot)$ denotes the $k$th permanental compound
of a matrix. Note that $\rho(\ext^k_{\per} (\pat A))\leq n!/(n-k)!$ for every matrix $A$.

Unlike in the case of polynomial roots, there is no universal
lower bound of $|\lambda_1 \dotsm \lambda_k|$ in terms of
$\gamma_1\dotsm \gamma_k$. However, we establish a lower
bound under additional assumptions (Theorem~\ref{lower_bound}).

\subsection{Related work}
The present inequalities generalize a theorem of Friedland,
who showed in~\cite{friedland}  that
for a nonnegative matrix $A$, we have 
\begin{align}
\rho_{\max}(A) \leq \rho(A) \leq \rho(\pat A)\rho_{\max}(A)
\label{e-shmuel}
\end{align}
where $\rho_{\max}(A)$ is the \emph{maximum cycle mean} of $A$, defined 
to be
\[
	\rho_{\max}(A) = \max_{i_1,\dots,i_\ell} {\left( a_{i_1,i_2} a_{i_2,i_3} \dotsm a_{i_\ell,i_1} \right)}^{1/\ell} \; ,
\]
the maximum being taken over all sequences $i_1,\dots,i_\ell$ 
of distinct elements of $[n]$.
Since $\gamma_1=\rho_{\max}(A)$, 
the second inequality in~\eqref{e-shmuel}
corresponds to the case $k=1$ in~\eqref{e-fried-gen}.

This generalization is inspired by a work of Akian, Bapat and Gaubert~\cite{abg04,abg04b}, dealing with matrices $A=(a_{i,j})$ over the field $\mathbb{K}$ of Puiseux series over $\mathbb{C}$ (and more generally, over fields of asymptotic expansions), equipped with its non-archimedean valuation $v$. 
It was shown in~\cite[Th.~1.1]{abg04} that generically, the valuations of the eigenvalues
of the matrix $A$ coincide with the tropical eigenvalues
of the matrix $v(A)$. Moreover, in the non-generic case, a majorization inequality still holds~\cite[Th.~3.8]{abg04b}. Here, we replace $\mathbb{K}$ by
the field $\mathbb{C}$ of complex numbers, and $v$ by the archimedean
absolute value $z\mapsto |z|$. Then, the majorization inequality
is replaced by a log-majorization inequality, up to a modification
of each scalar inequality by a multiplicative combinatorial constant,
and the generic equality is replaced by a log-majorization-type
lower bound, which requires
restrictive conditions. These results can be understood in the
light of tropical geometry, as the notion of tropical roots used here is
a special case of an amoeba~\cite{gelfand,passare,kapranov}.
Recall that if $\mathbb{K}$
is a field equipped with an absolute value $|\cdot|$, the amoeba of an algebraic
variety $V \subset (\mathbb{K}^*)^n$ is the closure of the 
image of $V$ by the map
which takes the log of the absolute value entrywise.   
Kapranov showed (see~\cite[Theorem 2.1.1]{kapranov})
that when $\mathbb{K}$ is a field equipped with a non-archimedean
absolute value, like the field of complex Puiseux series, the amoeba
of a hypersurface is precisely a {\em tropical hypersurface},
defined as the set of non-differentiability points of a certain convex piecewise affine map. The genericity result of~\cite{abg04,abg04b} can be reobtained
as a special case of Kapranov theorem, by considering the hypersurface
of the characteristic polynomial equation. 

When considering a field with an archimedean absolute value,
like the field of complex numbers equipped with its usual absolute
value, the amoeba of a hypersurface does not coincide any longer
with a tropical hypersurface, however, it can be approximated
by such a hypersurface, called spine,
in particular, Passare and Rullg{\aa}rd~\cite{passare} showed
that the latter is a deformation retract of the former.
In a recent work, Avenda\~{n}o, Kogan, Nisse and Rojas~\cite{rojas}
gave estimates of the distance between a tropical
hypersurface which is a more easily computable variant
of the spine, and the amoeba of a original hypersurface.  
However, it does not seem that the present bounds could be
derived by the same method.

We note that a different generalization of the Hadamard-Ostrowski-P\'olya theorem, dealing with the case of matrix polynomials, not relying on tropical eigenvalues, but thinking of the norm
as a ``valuation'', appeared recently in~\cite{ostags13}, refining
a result of~\cite{POSTA09}. Tropical eigenvalues generally
lead to tighter estimates in the case of structured or sparse matrices.

We finally refer the reader looking for more information on tropical linear algebra
to the monographs~\cite{baccelli,butkovic2010max}.

The paper is organized as follows. In Section~\ref{sec-prelim}, we recall
the definition and properties of tropical eigenvalues. The main result
(upper bound) is stated and proved in Section~\ref{sec-main}. 
The conditional lower bound is proved in Section~\ref{sec-lower}.
In Section~\ref{sec-compar}, we give examples in which the upper bound
is tight (monomial matrices) and not tight (full matrices), 
and compare the bound obtained by applying the present upper bound
to companion matrices with the original upper bound of Hadamard and P\'olya
for polynomial roots.

\section{Preliminary results}\label{sec-prelim}
\subsection{The additive and multiplicative models of the tropical semiring}
The \emph{max-plus semiring} $\maxplus$ is the commutative idempotent semiring
obtained by endowing the set $\mathbb {R}\cup\{-\infty\}$ with the
addition 
\(
a \oplus b = \max (a,b)\)
and the multiplication
\(
 a \odot b = a + b
\).
The zero and unit elements of the max-plus semiring are $\zero=-\infty$ and $\unit = 0$, respectively.

It will sometimes be convenient to work with a
variant of the max-plus semiring, the
\emph{max-times semiring} $\maxtimes$,
consisting of $\R_+$ (the set of nonnegative real numbers),
equipped with 
\(
a \oplus b = \max (a,b)\)
and \(
a \odot b = a \cdot b 
\),
so that the zero and unit elements are now 
$\zero=0$
and $\unit = 1$. Of course, the map $x\mapsto \log x$ is an
isomorphism $\maxtimes \to \maxplus$.
For brevity, we will often indicate
multiplication by concatenation, both
in $\maxplus$ and $\maxtimes$. The term {\em tropical
semiring} will refer indifferently to
$\maxplus$ or $\maxtimes$. Whichever
structure is used should always be clear from
the context.

\subsection{Tropical polynomials}
We will work with formal polynomials over a semiring $( \S, \oplus, \odot )$,
i.e.\ with objects of the form
\[
    p = \bigoplus_{k \in \N} a_k X^k, \qquad
a_k \in \S ,
\qquad 
\card{\left\{ k \mid a_k \neq \zero \right\}} < \infty.
\]

The set $\S[X]$ of all formal polynomials over $\S$ in the indeterminate $X$ is itself a semiring when endowed with the usual addition and multiplication (Cauchy product) of polynomials.

The {\em polynomial function}  $x\mapsto p(x), \S \to \S$, determined by the formal polynomial $p$, is defined by
\[
 p(x) = \bigoplus_{k=0}^n a_k \odot x^{\odot k} \; ,
\]
where $n$ is the degree of the polynomial.

When $\S$ is the max-times semiring $\maxtimes$ or the max-plus semiring $\maxplus$, the polynomial function becomes respectively
\[
p(x) = \max_{0 \leq k \leq n} a_k x^k, \qquad x \in \R_+
\]
or
\begin{align}
p(x) = \max_{0 \leq k \leq n} a_k + kx, \qquad x \in \R \cup \{-\infty\} \;  .
\label{e-def-evalp}
\end{align}

We will call \emph{max-plus polynomial} a polynomial over the semiring $\maxplus$,
and \emph{max-times polynomial} a polynomial over $\maxtimes$.
Both max-plus and max-times polynomials are referred to as \emph{tropical polynomials}.

Being an upper envelope of convex functions, tropical polynomial functions are convex and piecewise differentiable.
In particular, max-plus polynomial functions are piecewise linear.

\if{
\begin{figure}[h]
	\centering
	\subfloat[$p = -1X^3 \oplus 0X^2 \oplus 2X \oplus 1$]{
	\centering
	\includegraphics[width=.45\textwidth]{figures/polyfun1}
	}
	\subfloat[$p = -1X^3 \oplus -1X^2 \oplus 2X \oplus 1$]{
	\centering
	\includegraphics[width=.45\textwidth]{figures/polyfun2}
	}
\end{figure}}\fi

Different tropical polynomials can have the same associated polynomial function.
This lack of injectivity can be understood with the help of 
convex analysis results.
Let $f \colon \R \to \overline{\R}$ be an extended real-valued function.
The \emph{l.s.c.\ convex envelope} (or also \emph{convexification}) of $f$ is the function $\cvx f \colon \R \to \overline{\R}$ defined as
\[
 (\cvx f) (x) = \sup \left\{ g(x) \mid g \mbox{ convex l.s.c.}, \; g \leq f \right\} \;  .
\]
That is, $\cvx f$ is the largest convex l.s.c.\ minorant of $f$.
Analogously we define the \emph{u.s.c.\ concave envelope} (or \emph{concavification}) of $f$
as the smallest u.s.c.\ concave majorant of $f$, and we denote it by $\cav f$.

For any (formal) max-plus polynomial we define an extended real-valued function on $\R$ that represents its coefficients:
more precisely, to the max-plus polynomial ${p = \bigoplus_{k \in \N} a_k X^k}$ we associate the function
\[
 \coef p \colon \R \to \overline{\R} \qquad (\coef p)(x) = 
 \begin{cases}
  a_k &\mbox{if } x = k \in \N \\
  -\infty &\mbox{otherwise.}
 \end{cases}
\]
It is clear from~\eqref{e-def-evalp} that the max-plus polynomial
function $x\mapsto p(x)$ is 
nothing but the Legendre-Fenchel transform of the map $-\coef p$.

\begin{defn}
 Let $p = \bigoplus_{k=0}^{n} a_k X^k$ be a max-plus polynomial.
 The \emph{Newton polygon} $\Delta(p)$ of $p$ is the graph of the function $\cav (\coef p)$ restricted to the interval where it takes finite values.
 In other terms, the Newton polygon of $p$ is the upper boundary of the two-dimensional convex hull
 of the set of points $\left\{ (k, a_k) \mid 0 \leq k \leq n, a_k \neq -\infty \right\}$.
\end{defn}
The values $\cav (\coef p)(0), \dots, \cav (\coef p)(n)$ are called the \emph{concavified coefficients} of $p$,
and they are denoted by $\overline{a}_0, \dots, \overline{a}_n$,
or alternatively by $\cav (a_0),$ \dots, $\cav (a_n)$.
An index $k$ such that $a_k = \overline{a}_k$ (so that the point
$(k,a_k)$ lies on $\Delta(p)$) will be called
a \emph{saturated index}.
The polynomial $\overline{p} = \bigoplus_{k = 0}^{n} \overline{a}_k X^{k}$ is called the
\emph{concavified polynomial} of $p$.
The correspondence between a polynomial and its concavified is denoted by $\cav \colon p \mapsto \overline{p}$.

It is known that the Legendre-Fenchel transform of a map
depends only on its l.s.c.\ convex envelope; therefore,
we have the following elementary result.
\begin{prop}[{\cite[Chap~12, p.~104]{rockafellar}}]
Two max-plus polynomials have the same associated polynomial function
 if and only if they have the same concavified coefficients, or equivalently
the same Newton polygons.\qed
\end{prop}

Consider the isomorphism of semirings which sends a max-times polynomial  $p = \bigoplus_{k=0}^{n} a_k X^k$ to the max-plus
polynomial $\Log p:= \bigoplus_{k=0}^n (\log a_k)X^k$. 
We define the
\emph{log-concavified polynomial} of $p$
 as ${\widehat{p} = \Log^{-1} \circ \cav \circ \Log (p)}$.
 We also denote it by $\lcav p$; its coefficients are called the \emph{log-concavified coefficients} of $p$,
 and they are denoted by $\widehat{a}_0, \dots, \widehat{a}_n$,
 or alternatively by $\lcav (a_0), \dots, \lcav (a_n)$.

\subsection{Roots of tropical polynomials}
The \emph{roots} of a tropical polynomial are defined as the points of non-differentiability of its associated polynomial function. 
So, if $p = \bigoplus_{k=0}^n a_k X^k$ is a max-plus polynomial, then $\alpha \in \R\cup\{-\infty\}$ is a root of $p$ if and only if the maximum 
\[
\max_{0 \leq k \leq n} a_k + k\alpha
\]
is attained for at least two different values of $k$.
The \emph{multiplicity} of a root is defined
to be the difference between the largest and the smallest value of $k$ 
for which the maximum is attained. The same definitions
apply, mutatis mutandis, to max-times polynomials. In particular,
if $p$ is a max-times polynomial,
the tropical roots of $p$ are the images of the tropical roots of $\Log p$ by
the exponential map, and the multiplicities are preserved. 

Cuninghame-Green and Meijer~\cite{cuninghame80} showed
that a max-plus polynomial function of degree $n$,
$p(x) = \bigoplus_{k=0}^n a_k x^k$, can be factored uniquely
as 
\[
p(x) = a_n (x \oplus \alpha_1)\dots (x\oplus \alpha_n) \;  .
\]
The scalars $\alpha_1,\dots,\alpha_n$ are precisely the tropical
roots, counted with multiplicities.

Because of the duality arising from the interpretation of tropical polynomial
functions as Fenchel conjugates,
the roots of a max-plus polynomial are related to its Newton polygon.
The following result is standard.
\begin{prop}[See e.g.\ {\cite[Proposition 2.10]{abg04b}}]
 Let $p \in \maxplus [X]$ be a max-plus polynomial.
 The roots of $p$ coincide with the opposite of the slopes of the Newton polygon of $p$.
 The multiplicity of a root $\alpha$ of $p$ coincides with the length of the interval
 where the Newton polygon has slope $-\alpha$. 
\end{prop}
This proposition is illustrated in the following figure:
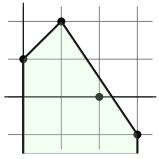
\begin{center}
\begin{tikzpicture}[scale=0.5]
  \draw[step=1cm,gray,very thin] (-0.4,-1.4) grid (3.4,2.4);
\draw[thick] (0,1) -- (0,-1.5); 
\draw[thick] (0,1) -- (1,2); 
\draw[thick] (1,2) -- (3,-1);
\draw[thick] (3,-1.5) -- (3,-1);
\fill (0,1) circle (3pt);
\fill (1,2) circle (3pt);
\fill (2,0) circle (3pt);
\fill (3,-1) circle (3pt);
\draw (-0.5,0) -- (3.5,0);
\draw (0,-1.5) -- (0,2.5);
  \filldraw[fill=green!20!white, draw=black, nearly transparent]
    (0,-1.5) -- (0,1) -- (1,2) -- (3,-1) -- (3,-1.5) -- cycle;
\end{tikzpicture}
\end{center}
Here, the polynomial is $p = -1X^3 \oplus 0X^2 \oplus 2X \oplus 1$. The tropical roots are $-1$ (with multiplicity $1$) and $1.5$ (with multiplicity $2$).

\begin{cor}
 \label{tropical_vieta}
 Let $p = \bigoplus_{k=0}^n a_k X^k$ be a max-plus polynomial, and let
 $\alpha_1 \geq \cdots \geq \alpha_n$ be its roots, counted with multiplicities.
 Then the following relation for the concavified coefficients of $p$ holds:
 \[
  \overline{a}_{n-k} = a_n + \alpha_1 + \dots + \alpha_k  \qquad \forall k \in 
[n].
 \]
 Analogously, if $p = \bigoplus_{k=0}^n a_k X^k$ is a max-times polynomial
 with roots ${\alpha_1 \geq \cdots \geq \alpha_n}$,
 then the following relation for its log-concavified coefficients holds:
 \[
  \widehat{a}_{n-k} = a_n \alpha_1 \dotsm \alpha_k \qquad \forall k \in [n].
 \]
\end{cor}
As pointed out in the introduction, the tropical roots, 
in the form of ``slopes of Newton polygons'',
were already apparent in the works of Hadamard~\cite{hadamard1893}
and 
 Ostrowski~\cite{Ostrowski1}.

We next associate a tropical polynomial to a complex polynomial.
\begin{defn}
 Given a polynomial $p \in \C[z]$,
\[
 p = \sum_{k=0}^{n} a_k z^k \;  ,
\]
we define its \emph{max-times relative} $p^{\times} \in \maxtimes [X]$ as
\[
 p^{\times} = \bigoplus_{k=0}^n |a_k| X^k \;  ,
\]
and its \emph{max-plus relative} $p^+ \in \maxplus [X]$ as
\[ p^+=\Log p^\times
\enspace .
\]
\end{defn}

We can now define the tropical roots of an ordinary polynomial.

\begin{defn}[Tropical roots]
 The \emph{tropical roots} of a polynomial $p \in \C[z]$ are the roots $\alpha \in \R_+$ of its max-times relative $p^{\times}$.
\end{defn}

\subsection{Tropical characteristic polynomial and tropical eigenvalues}
In this section we 
recall the definition
of the eigenvalues of a tropical matrix~\cite{abg04,abg04b}.
We start from the notion of permanent,
which is defined for matrices with entries
in an arbitrary semiring $( \S, \oplus, \odot )$.
\begin{defn}
	The \emph{permanent} of a matrix $A = (a_{i,j})
\in \S^{n \times n}$ is defined as
	\[ 
	\per_{\S} \colon \S^{n \times n} \to \S \;, \qquad
	\per_{\S} A = \bigoplus_{\sigma \in S_n} 
a_{1,\sigma(1)}\cdots a_{n,\sigma(n)}\;,%
	\]
where $S_n$ is the set of permutations of $[n]$.
\end{defn}
In particular, if $\S= \C$, 
	\[
	\per_\C A = \per A:= \sum_{\sigma \in S_n} \prod_{i = 1}^n a_{i,\sigma(i)}\; 
	\]
is the usual permanent, whereas if $S=\maxtimes$, 
we get
the {\em max-times permanent}
	\[ \per_{\maxtimes} A = \max_{\sigma \in S_n} \prod_{i = 1}^{n} a_{i, \sigma(i)} \; . \]
Computing the usual permanent is a difficult problem
which was indeed proved to be $\sharp$P-complete by Valiant~\cite{valiant}.
However, computing a max-times permanent
is nothing but solving 
an optimal assignment problem (a number of
polynomial time algorithms are known for this problem,
including $O(n^3)$ strongly polynomial algorithms,
see~\cite{burkard} for more background).
\begin{defn}
	The \emph{tropical characteristic polynomial} of a max-times matrix 
	${A \in \maxtimes^{n \times n}}$ is the tropical polynomial
	\[ q_A \in \maxtimes[X] \; , \qquad q_A = \per_{\maxtimes[X]} (A \oplus XI) \; ,\]
	where $I$ is the max-times identity matrix.
\end{defn}
\begin{exa*}
	\begin{align*}
	A = \begin{pmatrix} a & b \\ c & d \end{pmatrix}\;, \qquad
	A \oplus XI = \begin{pmatrix} a \oplus X & b \\ c & d \oplus X \end{pmatrix} \\
	q_A = (a \oplus X)(d \oplus X) \oplus bc = X^2 \oplus (a \oplus d)X \oplus (ad \oplus bc)
	\end{align*}
\end{exa*}
Note that an alternative, finer, definition~\cite{guterman}
of the tropical characteristic polynomial relies on {\em tropical determinants} which unlike permanents take into account signs.
This is relevant mostly when
considering real
eigenvalues, instead of complex
eigenvalues as we do here.

We can give explicit expressions for the coefficients of the tropical characteristic polynomial: if $A=(a_{i,j})\in \maxtimes^{n \times n}$ and we write 
$q_A = X^n \oplus c_{n-1}X^{n-1} \oplus \dots \oplus c_0$,
then it is not difficult to see that
\begin{align*}
	c_{n-k} &= \bigoplus_{\substack{I \subset [n] \\ \card{I} = k}} \bigoplus_{\sigma \in S_I} 
	a_{1, \sigma(1)} \dots a_{n,\sigma(n)}
= \bigoplus_{\substack{I \subset [n] \\ \card{I} = k}} \per_{\maxtimes} A[I,I]
\quad \forall k\in [n]\;,
\end{align*}
where $S_I$ is the group of permutations of the set $I$, and
$A[I,J]$ is the $k \times k$ submatrix obtained by selecting from $A$ the rows $i \in I$ and the columns $j \in J$.
It will be convenient to write the coefficients of $q_A$ in terms of the exterior powers of $A$.

\begin{defn}
The \emph{$k$-th exterior power} 
or \emph{$k$-compound} of a matrix $A \in \C^{n \times n}$ is the matrix
${\ext^k A \in \C^{\binom{n}{k} \times \binom{n}{k}}}$
whose rows and columns are indexed by the subsets of cardinality $k$ of 
$[n]$,
and whose entries are defined as
\begin{equation} \label{ext_pow}
	{\left( \ext^k A \right)}_{I,J} = \det A[I,J] \; .
\end{equation}
The $k$-th trace of $A$ is then defined as 
\[ \tr^k A = \tr \left( \ext^k A \right) = \sum_{\substack{I \subset [n] \\ \card{I} = k}} \det A[I,I] \]
for all $k \in [n]$.
If we replace the determinant with the permanent in Equation~\eqref{ext_pow},
we get the \emph{$k$-th permanental exterior power} of $A$, denoted by $\ext^k_{\per} A$.

Analogously, for a matrix $A \in \maxtimes^{n \times n}$, we define
the \emph{$k$-th tropical exterior power} of $A$ to be the matrix 
${\ext^k_{\maxtimes} A \in  \maxtimes^{\binom{n}{k} \times \binom{n}{k}}}$ whose entries are
\[ {\left( \ext^k_{\maxtimes} A \right)}_{I,J} = \per_{\maxtimes} A[I,J] \]
for all subsets $I, J \subset [n]$ of cardinality $k$.
The \emph{$k$-th tropical trace} of $A$ is defined as 
\begin{equation}\label{tropical_trace}
	\tr^k_{\maxtimes} A = \tr_{\maxtimes} \left( \ext^k_{\maxtimes} A \right) 
	= \max_{\substack{I \subset [n] \\ \card{I} = k}} \per_{\maxtimes} A[I,I] \; .
\end{equation}
\end{defn}

One readily checks that the coefficients of $q_A$
are given by $c_{n-k} = \tr^k_{\maxtimes} A$.

\begin{defn}[Tropical eigenvalues]
	Let $A \in \maxtimes^{n \times n}$ be a max-times matrix.
	The \emph{(algebraic) tropical eigenvalues} of $A$ are the roots
	of the tropical characteristic polynomial $q_A$.
	
	Moreover, we define the tropical eigenvalues of a complex matrix $A = (a_{i,j}) \in \C^{n \times n}$
	as the tropical eigenvalues of the associated max-times matrix $|A| = (|a_{i,j}|)$.
\end{defn}

\begin{rmk*}
No polynomial algorithm is known to compute all the coefficients of the tropical characteristic polynomial
(see e.g.~\cite{butkoviclewis}).
However, the \emph{roots} of $q_A$, only depend on the associated polynomial function, and can be computed by solving at most $n$ optimal assignment problems,
leading to the complexity bound of $O(n^4)$ of Burkard and Butkovi\v{c}~\cite{burkard-butkovic}. Gassner and Klinz~\cite{gassner} showed
that this can be reduced in $O(n^3)$ using parametric optimal assignment techniques.
\end{rmk*}

Before proceeding to the statement of our main result, we need a last definition. 

\begin{defn}
Denote by %
$\Omega_n$ the set of all cyclic permutations of %
$[n]$. 
For any $n \times n$ complex matrix $A = (a_{i,j})$
and for any bijective map $\sigma$ %
from a subset $I\subset[n]$ to a subset $J\subset[n]$,
we define the \emph{weight} of $\sigma$ with respect to $A$ as
\[
 w_A(\sigma) = \prod_{i\in I} a_{i,\sigma(i)} \;  .
\]
If $A$ is a nonnegative matrix, meaning a matrix with nonnegative real entries,
$\sigma$ is a permutation of $I$, and $I$ has cardinality $\ell$,
 we also define the \emph{mean weight} of $\sigma$ with respect to $A$ as
\[
 \mu_A(\sigma) = w_A(\sigma)^{1/\ell} = \Bigl(\prod_{i\in I} a_{i,\sigma(i)}\Bigr)^{1/\ell} \;  ,
\]
and  the \emph{maximal cycle mean} of %
$A$ as %
\[
 \rho_{\max} (A) = \max_{\sigma \in \Omega_n} \mu_A(\sigma) \;  .
\]
If we interpret $A$ as the adjacency matrix of a directed graph with weighted edges,
then $\rho_{\max}(A)$ represents the maximum mean weight of a cycle over the graph.
\end{defn}

\begin{rmk*}
Since any permutation can be factored into a product of cycles,
we can equivalently define the maximal cycle mean
in terms of general permutations instead of cycles:
\[
 \rho_{\max} (A) = \max_{1 \leq \ell \leq n} \max_{\card{I}=\ell} \;\max_{\sigma \in S_I} \mu_A(\sigma) \;  .
\]
\end{rmk*}

\begin{prop}[Cuninghame-Green, \cite{cuninghame83}]\label{rhoTequal}
 For any $A \in \maxtimes^{n \times n}$, the largest root $\rho_{\maxtimes} (A)$
of the tropical
 characteristic polynomial $q_A$ is equal to the maximal cycle mean $\rho_{\max} (A)$.
\end{prop}
We shall occasionally refer to  $\rho_{\maxtimes} (A)$
as the \emph{tropical spectral radius} of $A$.
\begin{rmk*}
The term {\em algebraic eigenvalue}
is taken from~\cite{abg05}, to which we refer
for more background. It is used there
to distinguish them from the \emph{geometric} tropical eigenvalues,
which are 
the scalars $\lambda \in \maxtimes$
such that there exists a non-zero vector $u \in \maxtimes^n$ (eigenvector)
such that
 $A \odot u = \lambda \odot u$. It is known that every 
geometric eigenvalue is an algebraic eigenvalue,
but not vice versa. Also, the maximal geometric eigenvalue
coincides with the maximal cycle mean $\rho_{\max} (A)$,
hence with the maximal algebraic eigenvalue, $\rho_{\maxtimes}(A)$.
\end{rmk*}

\section{Main result}\label{sec-main}

We are now ready to formulate our main result.

\begin{thm}
\label{upper_bound}
Let $A \in \C^{n \times n}$ be a complex matrix, and let $\lambda_1,\dots,\lambda_n$ be its eigenvalues,
ordered by nonincreasing absolute value (i.e., $|\lambda_1| \geq \dots \geq |\lambda_n|$).
Moreover, let $\gamma_1 \geq \dots \geq \gamma_n$ be the tropical eigenvalues of $A$.
Then for all $k \in [n]$, we have
\[
 |\lambda_1 \dotsm \lambda_k| \leq U_k \gamma_1 \dotsm \gamma_k
\]
where
\[
 U_k = \rho(\ext^k_{\per} (\pat A)) \;  .
\]
\end{thm}

To prove this theorem, we shall need some auxiliary results.

\subsection{Friedland's Theorem}
Let $A = \left( a_{i,j} \right)$ and $B = \left( b_{i,j} \right)$ be nonnegative matrices.
We denote by $A \circ B$ the Hadamard (entrywise) product of $A$ and $B$,
and by $A^{[r]}$ the entrywise $r$-th power of $A$. That is:
\[
 {(A\circ B)}_{i,j} = a_{i,j} b_{i,j}, \qquad {(A^{[r]})}_{i,j} = a_{i,j}^r \;  .
\]

\begin{thm}[Friedland, \cite{friedland}]
 \label{friedland_thm}
 Let $A$ be a nonnegative matrix. Define the \emph{limit eigenvalue} of $A$ as
 \[
  \rho_{\infty}(A) = \lim_{r \to +\infty} \rho(A^{[r]})^{1/r}.
 \]
 Then we have
 \begin{equation}
 \rho_{\infty}(A) = \rho_{\max}(A),
 \end{equation}
 and also
 \[
 \rho (A) \leq \rho(\pat A)\rho_{\max}(A),
 \]
 where $\pat A$ denotes the pattern matrix of $A$, defined as
 \[
  {(\pat A)}_{i,j} = 
   \begin{cases}
    0 &\mbox{if } a_{i,j} = 0\\
    1 &\mbox{otherwise}
   \end{cases}
 \]
\end{thm}

Friedland's result is related to the following log-convexity
property of the spectral radius.

\begin{thm}[Kingman~\cite{kingman}, Elsner, Johnson and Da Silva, \cite{elsner}]
\label{elsner_thm}
If $A$ and $B$ are nonnegative matrices,
and $\alpha, \beta$ are two positive real numbers such that $\alpha + \beta = 1$,
then $\rho(A^{[\alpha]} \circ B^{[\beta]}) \leq \rho(A)^\alpha \rho(B)^\beta$.
\end{thm}

\begin{cor}\label{cor-gen-friedland}
If $A$ and $B$ are nonnegative matrices, then $\rho(A \circ B) \leq \rho(A) \rho_{\max}(B)$.
\end{cor}
\begin{proof}
Let $p, q$ be two positive real numbers such that $\frac{1}{p} + \frac{1}{q} = 1$.
By applying Theorem~\ref{elsner_thm} to the nonnegative matrices $A^{[p]}$ and $B^{[q]}$, and $\alpha=\frac{1}{p}$,
we get $\rho(A \circ B) \leq \rho(A^{[p]})^\frac{1}{p} \rho(B^{[q]})^\frac{1}{q}$.
Then by taking the limit for $q \to \infty$ and using the identities of Theorem~\ref{friedland_thm}
we obtain $\rho(A \circ B) \leq \rho(A) \rho_{\max}(B)$.
\end{proof}

\subsection{Spectral radius of exterior powers}
The next two propositions are well known.
\begin{prop}[See e.g.\ {\cite[Theorem 8.1.18]{horn}}]
 \label{spectral_radius_inequalities}
 The following statements about the spectral radius hold:
 \begin{itemize}
  \item[(a)] For any complex matrix $A$ we have $\rho(A)\leq\rho(|A|)$;
  \item[(b)] If $A$ and $B$ are nonnegative matrices and $A\leq B$, then $\rho(A)\leq\rho(B)$.
 \end{itemize}
\end{prop}
\begin{prop}[See e.g.\ {\cite[2.15.12]{marcus}}]
 \label{rho_exterior_power}
 If $A \in \mathbb{C}^{n \times n}$ has eigenvalues $\lambda_1, \dots, \lambda_n$, then the eigenvalues of $\ext^k A$ are 
 the products $\prod_{i \in I} \lambda_i$ for all subsets $I \subset [n]$
of cardinality $k$.
\end{prop}

An immediate corollary of Proposition \ref{rho_exterior_power} is that if $|\lambda_1| \geq \dots \geq |\lambda_n|$, 
then the spectral radius of $\ext^k A$ is 
\[
 \rho(\ext^k A) = |\lambda_1 \dotsm \lambda_k| \;  .
\]
In the tropical setting we can prove the following combinatorial result,
which will be one of the key ingredients of the proof of Theorem \ref{upper_bound}.

\begin{thm}
 \label{rho_tropical_exterior_power}
 Let $A \in \C^{n \times n}$ be a complex matrix, and let $\gamma_1 \geq \cdots \geq \gamma_n$ be its tropical eigenvalues.
 Then for any $k \in [n]$ we have 
 \[
  \rho_{\maxtimes}(\ext^k_{\maxtimes} |A|) \leq \gamma_1 \dotsm \gamma_k.
 \]
\end{thm}

The proof of this theorem relies on the following result, which is
a variation on classical theorems of Hall and Birkhoff on doubly stochastic matrices. Recall that a {\em circulation matrix} of size $n\times n$ is a nonnegative matrix $B=(b_{i,j})$ such that for all $i\in [n]$, $\sum_{j\in[n]}b_{i,j}=\sum_{j\in [n]}b_{j,i}$. 
The {\em weight} of this matrix is the maximum value of the latter sums as $i\in [n]$. We call {\em partial permutation matrix} a matrix
having a permutation matrix as a principal submatrix, all the other entries
being zero.  The {\em support} of a partial permutation matrix consists of the row (or column) indices of this principal submatrix.
\begin{lem}\label{lem-circ}
Every circulation matrix $B=(b_{i,j})$ with integer entries, of weight $\ell$, can be written as the sum of at most $\ell$ partial permutation matrices.
\end{lem}
\begin{proof}
We set $s_i = \sum_{j\in[n]} b_{i,j} = \sum_{j\in [n]} b_{j,i}$,
so that $s_i \leq \ell \quad\forall i \in [n]$.
 If we add to $B$ the diagonal matrix $D=\Diag(\ell-s_1,\dots,\ell-s_n)$, we obtain a matrix with nonnegative integer entries
 in which the sum of each row and each column is $\ell$. 
A well known theorem (see e.g.\ Hall, {\cite[Theorem 5.1.9]{hall}}), allows us 
to write
 \[
 B+D = P^{(1)} + \dots + P^{(\ell)}
 \]
 where the $P^{(i)}$'s are permutation matrices.
 Furthermore we can write $D$ as a sum of diagonal matrices $D^{(1)}, \dots, D^{(\ell)}$ such that $D^{(i)} \leq P^{(i)} \;\forall i\in[\ell]$.
 In this way we have
 \[
 B = (P^{(1)}-D^{(1)}) + \dots + (P^{(\ell)}-D^{(\ell)}) = B^{(1)} + \dots + B^{(\ell)}
 \]
 where every $B^{(m)}=(b^{(m)}_{i,j})$ is a partial permutation matrix (possibly zero).
\end{proof}

\begin{proof}[Proof of Theorem~\ref{rho_tropical_exterior_power}]
Let $A \in \C^{n \times n}$ be a complex matrix.
By definition, the tropical eigenvalues $\gamma_1 \geq \cdots \geq \gamma_n$
of $A$ are the roots of the tropical characteristic polynomial 
$q_{|A|} = \per_{\maxtimes}(|A| \oplus XI  )$.
Recall that $\tr^k_{\maxtimes}|A|$ is the $(n-k)$-th coefficient of 
$q_{|A|}$, with the convention $\tr^0_\maxtimes |A| = 1$.
We shall denote by  $\widehat{\tr}^k_{\maxtimes}|A|$ the 
$(n-k)$-th log-concavified coefficient of $q_{|A|}$.

 In the following formulas we will denote by $S_{I,J}$ the set of bijections from $I$ to $J$.
 By Proposition~\ref{rhoTequal}, we have
 \begin{align}
 \rho_{\maxtimes}(\ext^k_{\maxtimes} |A|) &= 
 \max_{\ell\in [n]} \; \max_{\substack{\card{I_1} = k \vspace{-4pt} \\ \vspace{2pt} \vdots\\ \card{I_\ell} = k}} 
 {\left( {\ext^k_{\maxtimes} |A|}_{I_1I_2}\dotsm{\ext^k_{\maxtimes} |A|}_{I_lI_1} \right)}^{1/\ell}\nonumber\\
 &= \max_{\ell\in [n]} \; \max_{\substack{\card{I_1} = k \vspace{-4pt} \\ \vspace{2pt} \vdots\\ \card{I_\ell} = k}} \;
 \max_{\substack{\sigma_1\in S_{I_1,I_2} \vspace{-4pt} \\ \vspace{2pt} \vdots\\ \sigma_\ell\in S_{I_\ell,I_1}}}
 {\Bigl(\prod_{i_1\in I_1}|a_{i_1 \sigma_1(i_1)}|\dots \prod_{i_\ell\in I_\ell}|a_{i_\ell \sigma_\ell(i_\ell)}|\Bigr)}^{1/\ell}\enspace .\label{e-tocite}
 \end{align}
 The product in parentheses is a monomial in the entries of $|A|$ 
of degree $k\cdot \ell$. We rewrite it as
 \[
 \prod_{\substack{i \in [n]\\j \in [n]}}|a_{i,j}|^{b_{i,j}},
 \]
 where $b_{i,j}$ is the total number of times the element $|a_{i,j}|$ appears in the product. We can arrange the $b_{i,j}$ into a matrix
 $B=(b_{i,j})$, and observe that $\sum_{j\in [n]} b_{i,j} = \sum_{j\in [n]} b_{j,i}\; \forall i \in [n]$,
so that $B$ is a circulation matrix.
In fact, for every $m\in [\ell]$, every index $i\in I_m$ contributes for $1$ to the $i$-th row of $B$
 (because of the presence of $|a_{i,\sigma_m(i)}|$ in the product),
 and also for $1$ to the $i$-th column of $B$ (because of the presence of $|a_{\sigma_{m-1}^{-1}(i),i}|$ in the product). By Lemma~\ref{lem-circ}, we
can write $B = B^{(1)}+\cdots + B^{(r)}$ with $r\leq \ell$, where
$B^{(1)},\dots,B^{(r)}$ are partial permutation matrices, with
respective supports $I^{(1)},\dots,I^{(r)}$. We set $B^{(r+1)}=\dots = B^{(\ell)}=0$
and $I^{(r+1)}=\dots=I^{(\ell)}=\varnothing$.

The product in the definition of $\rho_{\maxtimes}(\bigwedge^k_{\maxtimes}|A|)$ 
(inside the parentheses in~\eqref{e-tocite}) can thus be rewritten as
 \[
 \begin{aligned}
 \prod_{\substack{i \in [n]\\j \in [n]}}|a_{i,j}|^{b_{i,j}}=
 &\prod_{m=1}^\ell \biggl(\prod_{\substack{i\in[n]\\j\in[n]}}{|a_{i,j}|}^{b^{(m)}_{i,j}} \biggr)\\
 \leq & \prod_{m=1}^\ell \tr^{\card{I^{(m)}}}_{\maxtimes}|A|\\
 \leq & \prod_{m=1}^\ell {\widehat{\tr}}^{\card{I^{(m)}}}_{\maxtimes}|A|\\
\leq & ({\widehat{\tr}}^k_{\maxtimes}|A|)^\ell,
 \end{aligned}
 \]
where the last inequality follows from the log-concavity of $k\mapsto {\widehat{\tr}}^k_{\maxtimes}|A|$ and from the fact that
 $\frac{1}{\ell}\sum_{m=1}^\ell\card{I^{(m)}} = k$.
So, using~\eqref{e-tocite}, we conclude that
$ \rho_{\maxtimes}(\ext^k_{\maxtimes}|A|)\leq {\widehat{\tr}}^k_{\maxtimes}|A|$.

 Now, ${\widehat{\tr}}^k_{\maxtimes}|A|$ is the $(n-k)$-th 
concavified coefficient of the tropical  polynomial $q_{|A|}$,
 whose roots are $\gamma_1\geq\dots\geq\gamma_n$.
 Applying Corollary~\ref{tropical_vieta}, and recalling that 
$\tr^{0}_{\maxtimes} |A| = 1$, we obtain
 \[
 {\widehat{\tr}}^k_{\maxtimes}|A| = \gamma_1\dotsm\gamma_k,
 \]
 so we conclude that
 \[
 \rho_{\maxtimes}(\ext^k_{\maxtimes}|A|) \leq \gamma_1\dotsm\gamma_k.
 \]
\end{proof}

\subsection{Proof of Theorem \ref{upper_bound}}
For all subsets $I,J$ of $[n]$, we have
\[
\begin{aligned}
{|\ext^k A|}_{I,J} &= |\det A[I,J]| \leq \per |A[I,J]| \\
&\leq \card{\{\sigma \in S_{I,J} | w_A(\sigma) \neq 0 \}} \cdot \max_{\sigma\in S_{I,J}} |w_A(\sigma)|\\
&={\left( \ext^k_{\per}  (\pat A)\right)}_{I,J}   {\left(\ext^k_{\maxtimes}  |A|\right)}_{I,J}.
\end{aligned}
\]
Since this holds for all $I$ and $J$, we can write, in terms of matrices,
\begin{align}
|\ext^k A| \leq \left(\ext^k_{\per} (\pat A)\right) \circ \left(\ext^k_{\maxtimes}  |A| \right) \;  .\label{e-circ}
\end{align}
We have
\begin{align*}
|\lambda_1\dots\lambda_k| &=
\rho(\ext^k  A) \; \qquad \text{(by Proposition~\ref{rho_exterior_power})}
\\&\leq \rho\bigl( (\ext^k_{\per} (\pat A))\circ(\ext^k_{\maxtimes}  |A|)\bigr) \;  \qquad \text{(by%
~\eqref{e-circ} and Proposition \ref{spectral_radius_inequalities})},\\
& \leq \rho(\ext^k_{\per} (\pat A)) \rho_{\maxtimes}(\ext^k_{\maxtimes}  |A|) \;
\qquad \qquad\text{(by Corollary~\ref{cor-gen-friedland}
and Proposition~\ref{rhoTequal})}\\
& \leq  \rho(\ext^k_{\per} (\pat A)) \gamma_1 \dotsm \gamma_k \;  \qquad 
\text{(by Theorem~\ref{rho_tropical_exterior_power})}
\end{align*}
and the proof of the theorem is complete.\qed

\section{Lower bound}\label{sec-lower}
We next show that the product of the $k$ largest absolute values of eigenvalues
can be bounded from below in terms of the $k$ largest tropical
eigenvalues, under some quite restrictive non-degeneracy conditions.
\begin{lem} \label{lower_bound_lemma}
 Let $A = (a_{i,j})\in\mathbb{C}^{n\times n}$ be a complex matrix, and let $\lambda_1,\dots,\lambda_n$ be its eigenvalues,
 ordered by nonincreasing absolute value (i.e.\ $|\lambda_1| \geq \dots \geq |\lambda_n|$).
 Moreover, let $\gamma_1 \geq \dots \geq \gamma_n$ be the tropical eigenvalues of $A$.
 Let $k\in[n]$
be a saturated index for the tropical characteristic polynomial $q_{|A|}$.
 Suppose $\tr^k A\neq 0$, and let $C_k$ be any positive constant such that
 \begin{equation} \label{hyp_lemma}
 C_k \tr^k_{\maxtimes} |A| \leq |\tr^k A| \; .
 \end{equation}
 Then the following bound holds:
 \[ \frac{C_k}{\binom{n}{k}} \gamma_1 \dotsm \gamma_k \leq |\lambda_1 \dotsm \lambda_k| \; . \]
\end{lem}
\begin{proof}
Thanks to Ostrowski's lower bound in~\eqref{e-ost},
we already have
\[ \alpha_1 \dotsm \alpha_k \leq \binom{n}{k} |\lambda_1\dotsm\lambda_k|, \]
where $\alpha_1\geq\dots\geq\alpha_n$ are the tropical roots
of the ordinary characteristic polynomial
\[ p_A(x) = \det (xI - A) = x^n - (\tr A) x^{n-1} + (\tr^2 A) x^{n-2} + \dots + (-1)^n \tr^n A \; . \]
Moreover, by Corollary~\ref{tropical_vieta} we have
\[ 
\begin{aligned}
	\alpha_1 \dotsm \alpha_k &= \lcav(|\tr^k A|) \\
	\gamma_1 \dotsm \gamma_k &= \lcav(\tr^k_{\maxtimes} |A|) \; ,
\end{aligned}
\]
and since $k$ is a saturated index for $q_{|A|}$, $\lcav(\tr^k_{\maxtimes} |A|) = \tr^k_{\maxtimes} |A|$.
Now we can use Equation~\eqref{hyp_lemma} and write
\[
\gamma_1 \dotsm \gamma_k = \tr^k_{\maxtimes}|A| \leq \frac{1}{C_k} |\tr^k A|
\leq \frac{1}{C_k} \lcav (|\tr^k A|) = \frac{1}{C_k} \alpha_1 \dotsm \alpha_k
\leq \frac{\binom{n}{k}}{C_k} |\lambda_1 \dotsm \lambda_k| \; .
\]
\end{proof}

\begin{thm} \label{lower_bound}
 Let $A$, $\lambda_1,\dots,\lambda_n$, $\gamma_1,\dots,\gamma_n$ be as
in Lemma~\ref{lower_bound_lemma}, and let $k$ be a saturated index
 for the tropical characteristic polynomial $q_{|A|}$.
 Suppose that among the subsets of cardinality $k$ of 
$[n]$
there is a unique subset $\overline{I}_k$
 for which there exists a (possibly not unique) permutation $\overline{\sigma}\in S_{\overline{I}_k}$
 that realizes the maximum
 \begin{equation} \label{max_perm_weight}
  \max_{\substack{I\subset [n]\\\card{I}=k}} \max_{\sigma \in S_I} \prod_{i \in I} |a_{i,\sigma(i)}|
 \end{equation}
 (that is, $w_{|A|}(\overline{\sigma})=\tr^k_{\maxtimes}|A|$).
 Suppose $\det A[\overline{I}_k,\overline{I}_k] \neq 0$.
 Finally suppose that, for any permutation $\sigma$ of any subset of cardinality $k$ except $\overline{I}_k$,
 $w_{|A|}(\sigma)\leq\delta_k\cdot w_{|A|}(\bar{\sigma})=\delta_k\tr^k_{\maxtimes}|A|$, with
 \[
  \delta_k < \frac{|\det A[\overline{I}_k,\overline{I}_k]|}{\tr^k_{\maxtimes}|A|\left(\binom{n}{k}-1\right)k!}.
 \]
 Then the inequality
 \[
  L_k\gamma_1\dotsm\gamma_k\leq|\lambda_1\dotsm\lambda_k|
 \]
 holds with
 \[
  L_k = \frac{1}{\binom{n}{k}}\left(\frac{|\det A[\overline{I}_k,\overline{I}_k]|}
  {\tr^k_{\maxtimes}|A|} - \delta_k\mbox{$\left(\binom{n}{k}-1\right)k!$} \right).
 \]
\end{thm}
\begin{proof}
To prove the theorem it is sufficient to show that \eqref{hyp_lemma} holds with 
\[ 
C_k = \left( \frac{|\det A[\overline{I}_k,\overline{I}_k]|}{\tr^k_{\maxtimes}|A|}
- \delta_k \left(\binom{n}{k}-1\right)k! \right) \; .
\]
We have
\begin{align*}
|\tr^k A| &= \Bigl|\sum_{\substack{I\in[n]\\\card{I}=k}}\det A[I,I]\Bigr|\\
&\geq \Bigl|\det A[\overline{I}_k,\overline{I}_k]\Bigr|-\Bigl|\sum_{\substack{\card{I}=k\\I\neq\overline{I}_k}}\det A[I,I]\Bigr|\\
&\geq \Bigl|\det A[\overline{I}_k,\overline{I}_k]\Bigr|-\sum_{\substack{\card{I}=k\\I\neq\overline{I}_k}}\per \bigl|A[I,I]\bigr|\\
&\geq \Bigl|\det A[\overline{I}_k,\overline{I}_k]\Bigr|-\left(\binom{n}{k}-1\right) k! \delta_k\tr^k_{\maxtimes}|A| \; ,\\
&\geq \left(\frac{|\det A[\overline{I}_k,\overline{I}_k]|}{\tr^k_{\maxtimes}|A|} - \delta_k\left(\binom{n}{k}-1\right)k! \right)
\tr^k_{\maxtimes}|A|\\
&=C_k\tr^k_{\maxtimes}|A|\; ,
\end{align*}
and the hypothesis on $\delta_k$ guarantees that $C_k > 0$.
\end{proof}

If the maximum in \eqref{max_perm_weight} is attained by exactly one permutation,
then the statement of Theorem~\ref{lower_bound} can be slightly modified as follows.
\begin{thm}
Let $A$, $\lambda_1,\dots,\lambda_n$, $\gamma_1,\dots,\gamma_n$ and $k$ be as 
in Theorem~\ref{lower_bound}.
 Suppose that the maximum in~\eqref{max_perm_weight} is attained for a unique
 permutation $\bar{\sigma}$, and that for any other permutation $\sigma$ of
 any $k$-subset of 
$[n]$
the inequality
 $\frac{w_{|A|}(\sigma)}{w_{|A|}(\bar{\sigma})} \leq \eta_k$
 holds for some 
 \[ \eta_k < \frac{1}{\left(\binom{n}{k} k! - 1 \right)} \; . \]
 Then the inequality
 \[ L_k\gamma_1\dotsm\gamma_k\leq|\lambda_1\dotsm\lambda_k| \]
 holds with
 \[ L_k = \frac{1}{\binom{n}{k}}\left(1 -\eta_k \left(\binom{n}{k} k! - 1 \right) \right) \; . \]
\end{thm}

\begin{proof}
The arguments of the proof are the same as for Theorem~\ref{lower_bound}.
In the present case, we have
\[
\begin{aligned}
|\tr^k A| &= \Bigl|\sum_{\substack{I\in[n]\\\card{I}=k}}\det A[I,I]\Bigr|\\
&\geq \Bigl|w_{|A|}(\bar{\sigma})\Bigr|-\Bigl|\sum_{\sigma\neq\bar{\sigma}} w_{|A|}(\sigma)\Bigr|\\
&\geq \tr^k_{\maxtimes} |A| - \left( \binom{n}{k} k! - 1 \right) \eta_k \tr^k_{\maxtimes} |A| \; ,
\end{aligned}
\]
and we conclude applying Lemma~\ref{lower_bound_lemma}.
\end{proof}

\section{Optimality of the upper bound and comparison with the bounds for polynomial roots}\label{sec-compar}
We now discuss briefly the optimality of the upper bound for some special classes of matrices. 
Throughout this paragraph, if $A$ is a complex $n \times n$ matrix, then 
$\lambda_1, \dots, \lambda_n$ will be its eigenvalues (ordered by nonincreasing absolute value),
and $\gamma_1 \geq \cdots \geq \gamma_n$ will be its tropical eigenvalues.

\subsection{Monomial matrices.}
Recall that a {\em monomial matrix} is the product of a diagonal matrix (with
non-zero diagonal entries) and of a permutation matrix. We next show
that
the upper bound is tight for monomial matrices.
\begin{prop}
If $A$ is a monomial matrix, then, for all $k\in[n]$,
the inequality in Theorem~\ref{upper_bound} is tight,
\begin{equation}
\label{eq_perm}
|\lambda_1\dotsm\lambda_k| = \rho(\ext^k_{\per} (\pat A)) \gamma_1\dotsm\gamma_k\qquad\forall k\in[n]
\end{equation}
\end{prop} 
\begin{proof}
We claim that if $A$ is a monomial matrix, then, the absolute values of the eigenvalues
of $A$ coincide with the tropical eigenvalues of $|A|$, counted
with multiplicities. 

To see this, assume that $A=DC$ where $D$
is diagonal and $C$ is a matrix representing a permutation
$\sigma$. If $\sigma$ consists of several cycles, then, $DC$
has a block diagonal structure, and so, the characteristic
polynomial of $A$ is the product of the characteristic
polynomials of the diagonal blocks of $A$. The same
is true for the tropical characteristic polynomial of $|A|$.
Hence, it suffices to show the claim when $\sigma$
consists of a unique cycle. Then, denoting by $d_1,\dots,d_n$
the diagonal terms of $D$, expanding the determinant of $xI-A$ or the permanent
of $xI\oplus A$, one readily checks that
the characteristic polynomial of $A$ is $x^n-d_1\dots d_n$,
whereas the tropical characteristic polynomial of $|A|$
is $x^n\oplus |d_1\dots d_n|$. It follows that the
eigenvalues of $A$ are the $n$th roots of $d_1\dots d_n$,
whereas the tropical eigenvalues of $|A|$ are all equal to $|d_1\dots d_n|^{1/n}$. So, the claim is proved.

It remains to show that $\rho(\bigwedge^k_{\per} (\pat A)) = 1$. Note that $\pat A=C$.
We claim that $\bigwedge^k_{\per} C$ is a permutation matrix.
In fact, for any fixed $k\in[n]$,
let $I$ be a subset of cardinality $k$ of $[n]$.
Since $C$ is a permutation matrix, there is one and only one subset $J\subset[n]$ such that
$\per C[I,J] \neq 0$: precisely, if $C$ represents the permutation $\sigma\colon[n]\to[n]$,
then $\per C[I,\sigma(I)] = 1$ and $\per C[I,J] = 0 \; \forall J \neq \sigma(I)$.
This means that each row of $\bigwedge^k_{\per} C$ contains exactly one $1$, and the other entries are zeroes.
Since the same reasoning is also valid for columns, we can conclude that $\bigwedge^k_{\per} C$ is a permutation matrix,
and as such its spectral radius is $1$.
\end{proof}

\subsection{Full matrices.}
Monomial matrices are among the sparsest matrices we can think of.
One may wonder what happens in the opposite case, when all the matrix entries are nonzero. We next discuss a class of instances of this kind,
in which the upper bound is not tight.
We only consider the case $k = n$ for brevity,
although it is not the only case for which the equality fails to hold.
\begin{prop}
Let $A = \bigl(a_{i,j}\bigr)$ be a $n\times n$ complex matrix, $n\geq 3$,
and suppose ${|a_{i,j}| = 1}$ for all $i,j \in [n]$.
Then the inequality in Theorem~\ref{upper_bound} can not be tight for $k = n$.
\end{prop}
\begin{proof}
For any couple $(I,J)$ of $k$-subsets of $[n]$, the $(I,J)$ element of the matrix $\bigwedge^k_{\per} (\pat A)$
is given by the permanent of a $k\times k$ matrix of ones, that is $k!$;
so $\bigwedge^k_{\per} (\pat A)$ is a $\binom{n}{k} \times \binom{n}{k}$ matrix with all entries equal to $k!$.
Its spectral radius is therefore $\binom{n}{k} k!$ (and $(1,\dots,1)^\top$ is an eigenvector for the maximum eigenvalue).
For $k = n$, $\rho(\bigwedge^k_{\per} (\pat A))$ reduces to $n!$,
so our upper bound would be $|\lambda_1,\dotsm,\lambda_n|\leq n!\cdot\gamma_1\dotsm\gamma_n$.
Now, the left-hand side can be thought of as $|\det A|$,
and on the other hand $\gamma_1=\dots = \gamma_n = 1$
(the tropical characteristic polynomial is 
$q_A(x) = x^n\oplus x^{n-1}\oplus\dots\oplus x\oplus 1 = x^n\oplus 1 = (x\oplus 1)^n$
$\forall x\geq 0$).
So the inequality in Theorem~\ref{upper_bound} is equivalent to $|\det A|\leq n!$.
But the well-known Hadamard bound for the determinant yields in this case $|\det A|\leq(\sqrt{n})^n = n^{n/2}$,
and since $n^{n/2} < n!\;\forall n\geq 3$ the inequality of Theorem~\ref{upper_bound} can not be tight.
\end{proof}

\subsection{Comparison with the Hadamard-P\'olya's bounds for polynomial roots.}
Finally, we discuss the behavior of the upper bound of Theorem~\ref{upper_bound} for the case of a companion matrix. Since the eigenvalues of a companion matrix are exactly the roots of its associated polynomial,
this will allow a comparison between the present matrix bounds and
the upper bound of Hadamard and P\'olya discussed
in the introduction. 
We start by showing that the usual property of companion matrices remains true in the tropical setting.
\begin{lem}
Consider the polynomial $p(x) = x^n + a_{n-1} x^{n-1} + \dots + a_1 x + a_0$, and let $A$ be its companion matrix.
Then the tropical eigenvalues of $A$ are exactly the tropical roots of $p$.
\end{lem}
\begin{proof}
The matrix is
\[
A = 
\begin{pmatrix}
	0 	& 1 	& 0 	& \cdots& 0	\\
	0 	& 0	& 1	& \cdots& 0	\\
	\vdots 	& \vdots& \vdots& \ddots& \vdots\\
	0 	& 0	& 0	& \cdots& 1	\\
	-a_0	&-a_1	&-a_2	& \cdots&-a_{n-1}
\end{pmatrix} \;  .
\]
By definition, its tropical eigenvalues are the tropical roots of the polynomial $q_A(x) = \bigoplus_{k=0}^n \tr_{\maxtimes}^k \!|A| \,\,x^{n-k}$,
so to verify the claim it is sufficient to show that $\tr_{\maxtimes}^k \!|A| = |a_{n-k}|$ for all $k\in\{0,\dots,n\}$.
Recall that $\tr_{\maxtimes}^k \!|A|$ is the maximal tropical
permanent 
of a $k \times k$ principal submatrix of $|A|$
(see Equation \eqref{tropical_trace}).
It is easy to check that the only principal submatrices 
with a non-zero contribution
are those of the form $|A|[I_k,I_k]$ with $I_k = \{n-k+1,\dots,n\}$,
and in this case
$\per_{\maxtimes}|A|[I_k,I_k]=
|a_{n-k}|$.
\end{proof}

\begin{lem}\label{lem-explicit}
If $A$ is the companion matrix of a polynomial of degree $n$, then, 
\[
 {\rho(\ext^k_{\per} (\pat A)) \leq \min (k+1, n-k+1)}
\enspace .
\]
\end{lem}
\begin{proof}
First, we note that nonzero entries of $\ext^k_{\per} (\pat A)$ can only be $1$'s, because $\pat A$ is a $(0,1)$-matrix, and the tropical permanent of any of its square submatrices has at most one non-zero term.
By computing explicitly the form of $\ext^k_{\per} (\pat A)$, 
for example following the method used by Moussa in~\cite{moussa},
we see that each column of $\ext^k_{\per} (\pat A)$ has either one or $k+1$ nonzero entries,
and each row has either one or $n-k+1$ nonzero entries.
In terms of matrix norms, we have ${\Vert \ext^k_{\per} (\pat A) \Vert}_1 = k+1$,
and ${\Vert \ext^k_{\per} (\pat A) \Vert}_\infty = n-k+1$.
Since both these norms are upper bounds for the spectral radius,
we can conclude that ${\rho(\ext^k_{\per} (\pat A)) \leq \min (k+1, n-k+1)}$.
\end{proof}

Thus, by specializing Theorem~\ref{upper_bound} to companion matrices, 
we recover the version of the upper bound~\eqref{e-ost} originally
derived by Hadamard, with the multiplicative constant $k+1$. By comparison,
the multiplicative constant in Lemma~\ref{lem-explicit} is smaller due
to its symmetric nature. However, it was observed by Ostrowski
that the upper bound in~\eqref{e-ost} can be strengthened
by exploiting symmetry. We give a formal argument
for the convenience of the reader.
\begin{lem}
	Let $P = \{ p \in \C[z] \mid \deg p =n\}$
	be the set of complex polynomials of degree $n$.
	Denote the roots and the tropical roots as above.
	Suppose that the inequality $|\zeta_1 \dotsm \zeta_k| \leq f(k)\cdot\alpha_1 \dotsm \alpha_k$
	holds for some function $f$, for all $k \in [n]$
	and for all polynomials $p \in P$.
	Then the inequality $|\zeta_1 \dotsm \zeta_k| \leq f(n-k) \cdot \alpha_1 \dotsm \alpha_k$
	also holds for all $k \in [n]$ and for all polynomials $p \in P$.
\end{lem}
\begin{proof}
Consider
a polynomial $p \in P, \; p(z) = a_n z^n + \dots + a_0$ 
	with roots $\zeta_1, \dots, \zeta_n$ (ordered by nonincreasing absolute value)
	and tropical roots $\alpha_1 \geq \dots \geq \alpha_n$.
Arguing by density, we may assume that $a_0\neq 0$. Then, 
we build its reciprocal polynomial $p^*(z) = z^n p(1/z) = a_0 z^n + \dots + a_n$.
	It is clear that the roots of $p^*$ are $\zeta_1^{-1}, \dots, \zeta_n^{-1}$.
	Moreover, its tropical roots are $\alpha_1^{-1} \leq \dots \leq \alpha_n^{-1}$:
	this can be easily proved by observing that the Newton polygon of $p*$
	is obtained from the Newton polygon of $p$ by symmetry with respect to a vertical axis,
	and thus it has opposite slopes.

	Since $p^* \in P$, by hypothesis we can bound its $n-k$ largest roots:
	\[
		\left| \frac{1}{\zeta_n} \dotsm \frac{1}{\zeta_{k+1}} \right|
		\leq f(n-k) \cdot \frac{1}{\alpha_n} \dotsm \frac{1}{\alpha_{k+1}} \; .
	\]
	By applying Corollary~\ref{tropical_vieta} (and observing that 
$0$ is a saturated index
for the max-times relative $p^\times$) we also have
	\[ |a_0| = |a_n| \alpha_1 \dotsm \alpha_n \; ,\]
	so we can write
	\[
	\begin{aligned}
		|\zeta_1 \dotsm \zeta_k| &= |\zeta_1 \dotsm \zeta_n| 
		\left| \frac{1}{\zeta_n} \dotsm \frac{1}{\zeta_{k+1}} \right| \\
		&= \frac{|a_0|}{|a_n|} \left| \frac{1}{\zeta_n} \dotsm \frac{1}{\zeta_{k+1}} \right| \\
		&\leq \frac{|a_0|}{|a_n|} \cdot f(n-k) \cdot \frac{1}{\alpha_n} \dotsm \frac{1}{\alpha_{k+1}} \\
		&= \alpha_1 \dotsm \alpha_n \cdot f(n-k) \cdot \frac{1}{\alpha_n} \dotsm \frac{1}{\alpha_{k+1}} 
		= f(n-k) \cdot \alpha_1 \dotsm \alpha_k
	\end{aligned}
	\]
\end{proof}

Therefore, it follows from the P\'olya's upper bound~\eqref{e-ost} that
 \[
  |\zeta_1 \dotsm \zeta_k| \leq \min\left(\sqrt{\frac{{(k+1)}^{k+1}}{k^k}},
  \sqrt{\frac{{(n-k+1)}^{n-k+1}}{{(n-k)}^{n-k}}} \right) \alpha_1 \dotsm \alpha_k,
 \]
 for all $k \in [n]$.
This is tighter than the bound derived from Theorem~\ref{upper_bound}
and Lemma~\ref{lem-explicit}. In the latter lemma, we used
a coarse estimation of the spectral radius, via norms. A finer
bound can be obtained by computing the true
spectral radius of $\ext^k_{\per} (\pat A)$ for the companion matrix $A$,
but numerical experiments indicate this still does not improve P\'olya's bound. This is perhaps not
surprising as the latter is derived by analytic functions techniques 
(Jensen inequality and Parseval identity), which do not naturally
carry over to the more general matrix case considered here.

\end{document}